\documentclass[12pt]{article}
\usepackage[a4paper, left=2cm, right=2cm, bottom=3cm, top=3cm]{geometry}
\usepackage{amsmath}
\usepackage{srcltx}
\usepackage{amsfonts}
\usepackage{amssymb}
\usepackage{amsthm}
\usepackage{lipsum}
\usepackage{psfrag}
\usepackage{verbatim}
\usepackage{graphicx} 
\usepackage{color}
\usepackage{faktor}
\providecommand{\U}[1]{\protect\rule{.1in}{.1in}}

\newtheorem{theorem}{Theorem}[section]

\newtheorem{corollary}[theorem]{Corollary}

\newtheorem{definition}[theorem]{Definition}

\newtheorem{proposition}[theorem]{Proposition}
\newtheorem{remark}[theorem]{Remark}

\newcommand{\diag}{\mathop{\rm diag}\nolimits}

\newcommand{\R}{\ensuremath{\mathbb{R}}}

\begin{document}
\title{Observability on the classes of non-nilpotent solvable three-dimensional Lie groups}
\author{Victor Ayala\\ Instituto de Alta Investigación, Universidad de Tarapac\'a \\ Arica, Chile \\vayala@academicos.uta.cl
\and Thiago Matheus Cavalheiro\\Departamento de Matem\'{a}tica, Universidade Estadual do Norte do Paraná \\ Jacarezinho, Brazil \\thiago\_cavalheiro@hotmail.com 
\and Alexandre J. Santana\\Departamento de Matem\'{a}tica, Universidade Estadual de Maring\'{a}\\ Maring\'a, Brazil \\ajsantana@uem.br\\ 
}
\maketitle

%%%%%%%%%%%%%%%%%%%%%%%%%%%%%%%%%%%%%%%%%%%%%%%%%%%%%%%%%%%%%%%%%%%%%%%%%%%%%%%%%%%%%%%

%%%%%%%%%%%%%%%%%%%%%%%%%%%%%%%%%%%%%%%%%%%%%%%%%%%%%%%%%%%%%%%%%%%%%%%%%%%%%%%%%%%%%%%

%%%%%%%%%%%%%%%%%%%%%%%%%%%%%%%%%%%%%%%%%%%%%%%%%%%%%%%%%%%%%%%%%%%%%%%%%%%%%%%%%%%%%%%

\begin{abstract}

In control theory, researchers need to understand a system's local and global behaviors in relation to its initial conditions. When discussing observability, the main focus is on the ability to analyze the system using an output space defined by an output map.

In this study, our objective was to establish conditions for characterizing the observability properties of linear control systems on Lie groups. We will focus on five classes of solvable, non-nilpotent three-dimensional Lie groups, examining local and global perspectives. This analysis explores the kernels of homomorphisms between the state space and its simply connected subgroups, where the output is projected onto the quotient space.
\end{abstract}

\section{Introduction}

In the 1960s, Kalman significantly advanced in studying control systems within Euclidean spaces. His research focused on essential topics such as controllability, observability, and stability, which are crucial to understanding the dynamics of control systems and their interrelationships. These areas laid the foundation for modern control theory and opened the door to advancements in various applications in engineering and technology.

As a main example of this era, consider the linear control system 
\begin{equation*}\Sigma_{\R^n}:\left\{
    \begin{array}{ccl}
         \dot{x} &=& Ax + Bu,  \\
         y &=& Cx, 
    \end{array}\right.
\end{equation*}
where $A \in \R^{n \times n}, B \in \R^{n \times m}$, $C \in \R^{n \times l}$, with $l < n$, and $\mathcal{U} =L_{l o c}^{1} (\Omega )\text{,}$ is the set of admissible controls, that is, the family of locally integrable functions $u :[0,T_{u}] \rightarrow \Omega  \subset $$\mathbb{R}^{m}$. The set $\Omega $ is closed and $0 \in i n t (\Omega )$. If $\Omega  =\mathbb{R}^{m}\text{,}$ the system is called unrestricted. Otherwise, $\Sigma _{G}$\ is restricted, \cite{sontag}.

Starting with an initial condition $x_0\in\mathbb{R}^n$ and a specific control $u\in\mathcal{U}$, we can fully describe the solution to the system as follows:

$$\phi^u_t(x_0)=e^{tA}\left(x_0+\int^t_0e^{-\tau AB}u(\tau) d\tau\right),$$

which satisfies the Cauchy problem with initial value $\dot{x}=Ax+Bu$, $x(0)=x_0$. Thus, $\phi^u_t(x_0)$ with $t\in\mathbb{R}$ describes a curve in $\mathbb{R}^d$ such that starting from $x_0$ the elements on the curve are reached from $x_0$ forward and backward through the specific dynamics of the linear system determined by the control.\\

Next, we introduce the concept of unobservability, which allows us to decompose the state space into equivalence classes. These classes group elements that cannot be distinguished from each other by the observation function \(h\) through the system dynamics.
 \begin{definition}
 Two states $x_0$, $x_1 $$ \in \mathbb{R}^n$ are indistinguishable by $(\Sigma_{\R^n})$, denoted by $x_0 I x_1$, if
 \[x_1-x_0\in ker(Ce^{tA}),\forall t\geq0.\]
 \end{definition}

\begin{remark}In fact, under the definition´s condition the observation function \( h \) does not differentiate between states \( x_0 \) and \( x_1 \) across the system, resulting in identical outcomes for each control $u\in\mathcal{U}$ and for any $t>0,$
  \[C\left( e^{tA}(x_0+\int_0^t e^{-\tau A}Bu(\tau)d\tau)\right)=C\left(e^{tA}(x_1+\int_0^t e^{-\tau A}Bu(\tau)d\tau)\right).\]
It turns out that,
  \begin{proposition}
  \begin{enumerate}
  Let us consider a linear control system on the Euclidean space   $\mathbb{R}^n$. Therefore, 
  \item $I$  is an equivalence relation
  \item If $I(x)$ denotes the equivalence class $x$ by the relation $I$, then,
  \begin{itemize}
 \item[a)] $I(0)=\bigcap_{t\geq 0}^{n-1} ker(Ce^{tA})$
 \item[b)] $I(x)=x+I(0)$.
  \end{itemize}
 \end{enumerate}
 \end{proposition}

So, the system is said to be observable if the equivalence class of the origin is trivial.
\end{remark}

From the previous analysis, the system $(\Sigma_{\R^n})$ is observable if and only if the matrix

\[
 \mathcal{O}=\left( \begin{array}{rrrrr}
C & CA & CA^2&\cdots& CA^{n-1} \end{array} \right)^T,
 \]

has maximum rank. See, for example, \cite{sontag}.

This concept was extended to Lie groups in
\cite{Ayalaetal}, see also \cite{AyalaHaci}.
The fundamental idea remained unchanged, but these two papers' findings were significantly more straightforward than the original. Furthermore, they provide necessary and sufficient conditions, which makes the reasoning even easier to follow. Currently, this concept is frequently applied in nonlinear ordinary differential equation ODE systems \cite{Aouadi}, control networks \cite{JiangWang}, dynamical systems \cite{Sivalingam}, and probabilistic systems \cite{ZhouGuoLiu}. 

Precisely, let $G$ be a Lie group with Lie algebra $\mathfrak{g}$.
$A\; l i n e a r\; c o n t r o l\; s y s t e m\; \Sigma _{G}$ on $G$ is determined by the family,
\begin{equation*}\Sigma _{G} :\text{}\overset{ \cdot }{g (t)}\text{} =\mathcal{X} (g (t)) +\sum _{j =1}^{m}u_{j} (t) Y^{j} (g (t)) ,\text{}g (t) \in G ,\text{}t \in \mathbb{R}\text{} ,u \in \mathcal{U}\text{,}
\end{equation*}of ordinary differential equations parametrized by the class $\mathcal{U} =L_{l o c}^{1} (\Omega )\text{,}$ as before.

The drift $\mathcal{X}$ is a linear vector field, meaning that its flow is a one-parameter group of automorphisms of the group, exactly as in the Euclidean space. And, for any $j\text{,}$ the control vector $Y^{j} \in \mathfrak{g}$, is considered as left-invariant vector field. We observe that any column vector of the matrix $B$ induces an invariant vector field on the Abelian Lie group $\mathbb{R}^n$.

From the orbit theorem of Sussmann, \cite{sontag}, without lost of generality we assume $\Sigma _{G}$ satisfy the Lie algebra rank condition (\emph{$L A R C)$}, which means: for any $g \in G\text{,}$\begin{equation}S p a n_{\mathcal{L} A} \left \{\mathcal{X} ,Y^{1} ,\ldots  ,Y^{m}\right \}(g) =T_{g} G\text{.}
\end{equation}Denote by $\varphi  (g ,u ,t)$ the solution of $\Sigma _{G}$ associated to the control $u$ with initial condition $g$ at the time $t\text{.}$ It turns out that, \cite{AyalaandAdriano} \begin{equation}\varphi  (g ,u ,t) =\mathcal{X}_{t} (g) \varphi  (e ,u ,t)\text{.}
\end{equation}
Thus, to compute the system's solution on $G$ through an initial condition $g$, we need to translate the solution through the identity element by the flow of the linear vector field acting on $g\text{.}$ 

Just observe the symmetry with the solution of a classical linear system on Euclidean spaces
\begin{equation}\phi  (x ,u ,t) =e^{t A} \left (x +\int \nolimits_{0}^{t}e^{ -\tau  A}\text{}B\text{}u \left (\tau \right ) d \tau \right )\text{.}
\end{equation}

This symmetry helps us characterize the observability properties on the group $G$ by applying concepts from classical linear systems in $\mathbb{R}^n$. However, it is necessary to introduce the notion of local observability, defined by a distribution determined by the Lie algebra associated with the indistinguishable class of the identity element $e$ of the group.

In this work, we utilize the classification of non-nilpotent solvable Lie groups outlined in \cite{onish} and the construction of linear systems on three-dimensional affine Lie groups described in \cite{AyalaandAdriano}. By applying the results from \cite{Ayalaetal}, we aim to list every simply connected subgroup of these groups and develop a general expression for the homomorphisms between the groups and the relevant sets. Finally, we establish the conditions under which the system is locally observable and observable.

\section{Preliminaries and Initial Results}

Let $G$ be a finite-dimensional Lie group with Lie algebra $\mathfrak{g}$. We use the results in \cite{Ayalaetal} and \cite{AyalaHaci} to characterize observability in our context, i.e., on a general pair coming from the drift and a projection onto a homogeneous space of $G$. In the classical linear control system, observability depends on the drift and the linear output map between finite dimensional vector spaces.
We start with the definition of the drift.

\begin{definition} A vector field $\mathcal{X}$ in $G$ is said to be linear if its flow $(\varphi_t)_{t \in \R}$ is a $1-$parameter subgroup of $\hbox{Aut}(G)$, the Lie group of automorphisms of $G$.
\end{definition}

In addition, to every linear vector field $\mathcal{X}$ of $G$, there is a derivation $\mathcal{D} \in \hbox{Der}(\mathfrak{g})$, satisfying 
\begin{equation*}
    -\mathcal{D}(Y) = [\mathcal{X},Y],
\end{equation*}
for every $Y \in \mathfrak{g}$. Recall that a derivation is a linear map that satisfies the Leibniz rule concerning the Lie bracket. The relationship between $\varphi_t$ and $\mathcal{D}$ is given by the formula \begin{equation*}
    d(\varphi_t)_e = e^{t\mathcal{D}}. 
\end{equation*}

\begin{remark}By the general theory of Lie groups (see \cite{SanMartin}, for instance), any closed subgroup $K$ induces a well-defined homogeneous space $G/K$. Therefore, it is possible to consider a canonical projection  $\pi_K: G \longrightarrow G/K$ as an output map.    
\end{remark}

We are willing to introduce the primary definition see \cite{AyalaHaci}. 

\begin{definition}A pair $(\mathcal{X}, \pi_K)$ in $G$ is determined by a linear vector field $\mathcal{X} $ and by a closed subgroup $K$ of $G$.     
\end{definition}

The main focus of this paper is the observability properties on Lie groups.

\begin{definition}A pair $(\mathcal{X},\pi_K)$ is said to be:
\begin{itemize}
    \item[1.] observable at $x_1$ if for all $x_2 \in G\setminus\{x_1\}$, there exists a $t \geq 0$ such that, $$\pi_K(\varphi_t(x_1)) \neq \pi_K(\varphi_t(x_2))$$
    \item[2.] locally observable at $x_1$ if there exists a neighborhood of $x_1$ such that the condition $1.$ is satisfied for each $x$ in the neighborhood.
    \item[3.] observable (locally observable) if it is observable (locally observable) for every $x \in G$. 
\end{itemize}
\end{definition}

In other words, two states $g_1, g_2\in G$ are said to be indistinguishbles if:
\[\mathcal{X}_t(g_1g_2^{-1})\in K\text{ }\forall t\geq 0.\]
Furthermore,

\[I=\{g\in G: \mathcal{X}_t(g)\in K,\forall t \in \R\}\]

is a closed normal subgroup of $G$, and the equivalence class of $g$ is $Ig.$

Next, we establish the main results to apply in our context (see \cite[Theorem 2.5]{AyalaHaci}).\\

The pair $(\mathcal{X}, \pi_K)$ in $G$ is observable if and only if, the Lie group $I$ is discrete and, $$Fix(\varphi) \cap K = \{e\}.$$ 

Notice that the first condition indicates local observability, characterized by the trivialization of the Lie algebra of $I$.

We conclude this section by proving a couple of propositions that will be useful later on.

\begin{remark}\label{samekernel} Consider $G$ a connected Lie group and $H_1,H_2 \subset G$ subgroups of $G$. Assume that $h_i: G \longrightarrow H_i$ for $i = 1,2$ are two homomorphisms whose respective kernels are given by $K_1$ and $K_2$. If $K_1 = K_2$, then $(\mathcal{X}, \pi_{K_1})$ is observable if, and only if, $(\mathcal{X}, \pi_{K_2})$ is observable. 

The reasoning behind showing this statement is direct. Since for $i=1,2,$  $$I_i = \{p \in G: \varphi_t(p) \in K_i, \forall t \in \R\},$$  it turns out that $I_1 = I_2$. And, $Fix(\varphi) \cap K_1 = Fix(\varphi) \cap K_2$.
\end{remark}

Consider $\mathcal{X}$ and $\mathcal{Y}$ linear vector fields defined over a Lie group $G$ with respective flows $\varphi$ and $\psi$. We say that $\mathcal{X}$ and $\mathcal{Y}$ are \textit{conjugated (or $\pi-$conjugated)} if there is an automorphism $\pi: G \longrightarrow G$ such that 
\begin{equation}\label{conjug}
\begin{array}{ccc}
     \pi(\varphi_t(x)) &=& \psi_t(\pi(x)), \forall t \in \R\\
    \varphi_t(\pi^{-1}(x)) &=& \pi^{-1}(\psi_t(x))), \forall t \in \R.  
\end{array}
\end{equation}

The following proposition will be helpful for our purposes. 

\begin{proposition}\label{conjug1}Consider a $\pi-$conjugation between the linear vector fields $\mathcal{X}$ and $\mathcal{Y}$ on $G$ and $h: G \longrightarrow G$ a homomorphism. If $K = \ker(h)$ and $S = \ker(h \circ \pi^{-1})$, then $(\mathcal{X}, \pi_K)$ is observable if, and only if, $(\mathcal{Y}, \pi_S)$ is observable. 
\end{proposition}

\begin{proof} As before, we denote the indistinguishable classes of the identity element 
\begin{equation*}
    I_1 = \{x \in G: \varphi_t(x) \in K, \forall t \in \R\}. 
\end{equation*}
and 
\begin{equation*}
    I_2 = \{y \in G: \psi_t(y) \in S, \forall t \in \R\}. 
\end{equation*}

Taking $x \in I_2$, we get $\psi_t(x) \in S$, for all $t \in \R$. Also, there is a $y \in G$ such that $\pi(y) = x$. Therefore, 
\begin{equation*}
    \psi_t(x) = \psi_t(\pi(y)) = \pi(\varphi_t(y)), \forall t \in \R, 
\end{equation*}
with $(h \circ \pi^{-1})(\psi_t(x) ) = e$, for all $t \in \R$. Applying $h \circ \pi^{-1}$ in the equality above, we get 
\begin{equation*}
    (h \circ \pi^{-1})(\psi_t(x) )=(h \circ \pi^{-1})(\pi(\varphi_t(y)))  = h (\varphi_t(y)) = e.  
\end{equation*}
which means that $\varphi_t(y) \in K$, for all $t \in \R$, that is, $y \in I_1$. This shows that $\pi^{-1}(I_2) \subset I_1$. 

On the other hand, if we consider $z \in I_1$, by the second expression in (\ref{conjug}), we obtain an element $w \in G$ such that $\pi^{-1}(w) = z$. Since $\varphi_t(z) = \varphi_t(\pi^{-1}(w)) \in  K$, we get: 
\begin{equation*}
    \varphi_t(\pi^{-1}(w)) = \pi^{-1}(\psi_t(w)) \in K, \forall t \in \R. 
\end{equation*}

Aplying $h$ in the expression above, we obtain $\psi_t(w) \in S$, for all $t \in \R$, which also give us $w
\in I_2$. This also proves that $\pi(I_1) \subset I_2$, which allow us to conclude that $\pi(I_1) = I_2$. 

Now, let us suppose that $I_1$ is discrete.  If $y \in I_2$, considering $x \in I_1$ such that $\pi(x) = y$ there is a neighborhood $V$ of $x$ such that $V \cap I_1 = \{x\}$. Therefore 
\begin{equation*}
    \{\pi(x)\} = \pi(V \cap I_1)  = \pi(V) \cap \pi(I_1) = \pi(V) \cap I_2 = \{y\}, 
\end{equation*}
with $\pi(V)$ neighborhood of $y$. We can conclude that $I_1$ is discrete if, and only if, $I_2$ is discrete, which proves the equivalence concerning local observability. 

Conjugations preserve fixed points. As a matter of fact, if $\varphi_t(g) = g$ for all $t \in \R$, then 
\begin{equation*}
    \pi(\varphi_t(g)) = \psi_t(\pi(g)) = \pi(g), \forall t \in \R. 
\end{equation*}

The same follows for the fixed points of $\psi$ using the function $\pi^{-1}$. \\
Now, if $Fix(\varphi) \cap K = \{e\}$, take $q \in Fix(\psi) \cap S$. Therefore, $\psi_t(q) = q,$ for all $t \in \R$. We obtain 
\begin{equation*}
    \pi^{-1}(\psi_t(q)) = \varphi_t(\pi^{-1}(q)) = \pi^{-1}(q), \forall t \in \R. 
\end{equation*}

Thus, $\pi^{-1}(q) \in Fix(\varphi)$. As $(h \circ \pi^{-1})(q) = e$, we get $\pi^{-1}(q) \in K$. Then $\pi^{-1}(q) \in K \cap Fix(\varphi) = \{e\}$, implying $q = e$. The same reasoning can be applied using the function $\pi$ at an arbitrary point $g \in Fix(\varphi) \cap K$. 
\end{proof}

From the proof of Proposition 2.6, we can easily derive the following result. 
\begin{corollary}With the same hypothesis of the previous proposition, $I_1$ is discrete if, and only if, $I_2$ is discrete. Also, $Fix(\varphi) \cap K = \{e\}$ if, and only if, $Fix(\psi) \cap S = \{e\}$.      
\end{corollary}

\section{Affine 3-dimensional Lie group}

In this section, we explore the observability properties of linear control systems within five classes of solvable, non-nilpotent three-dimensional Lie groups, examining both local and global perspectives. According to the classification provided by  \cite{onish}, the following list presents these groups up to isomorphism:

\begin{itemize}
    \item[1-] $\mathfrak{r}_2 = \R \times_{\theta} \R^2, \theta = \begin{bmatrix}
        0 & 0 \\
        0 & 1
    \end{bmatrix}$.
    \item[2-] $\mathfrak{r}_3 = \R \times_{\theta} \R^2, \theta = \begin{bmatrix}
        1 & 1 \\
        0 & 1
    \end{bmatrix}$.
    \item[3-] $\mathfrak{r}_{3,\lambda} = \R \times_{\theta} \R^2, \theta = \begin{bmatrix}
        1 & 0 \\
        0 & \lambda
    \end{bmatrix}, \hbox{ with } |\lambda| \in (0,1].$
    \item[4-] $\mathfrak{r}_{3,\lambda}' = \R \times_{\theta} \R^2, \theta = \begin{bmatrix}
        \lambda & -1 \\
        1 & \lambda
    \end{bmatrix}, \hbox{ with }\lambda \in \R \setminus\{0\}$.
    \item[5-] $\mathfrak{e} = \R \times_{\theta} \R^2, \theta = \begin{bmatrix}
        0 & -1 \\
        1 & 0
    \end{bmatrix}.$ 
\end{itemize}

All of these algebras can be described as a semi-direct product. The corresponding simply connected Lie groups  $R_2,R_3,R_{3,\lambda}, R'_{3,\lambda}$ and $E$, are constructed through the semi-direct product  $\R\times_{\rho} \R^2$, with $\rho_t = e^{t\theta}$.

\subsection{Linear Vector Fields}

We will begin by demonstrating some properties of the drift. For each $s \in \R$, define $\Lambda_s$ by
\begin{equation*}
    \Lambda_s =\begin{bmatrix}
        s & 0 \\
        0 & e^s-1
    \end{bmatrix}.
\end{equation*}

Consider a linear vector field $\mathcal{X}$ on $G$ and $\mathcal{D}$ the associated derivation. By \cite[Proposition 3.4]{AyalaandAdriano}, there is a linear transformation $\mathcal{D}^*: \R^2 \longrightarrow \R^2$ such that 
\begin{equation*}
    \mathcal{D}(0,v) = (0,\mathcal{D}^*v). 
\end{equation*}

Moreover, if $G$ is simply connected, the vector field $\mathcal{X}$ reads as 
\begin{equation*}
    \mathcal{X}(t,v) = (0,\mathcal{D}^*v + \Lambda_t\xi),
\end{equation*}
with $(0,\xi) = \mathcal{D}(1,0)$. By \cite[Remark 3.2]{AyalaandAdriano}, the solution $\varphi$ is defined through the formula,
\begin{equation*}
    \varphi_s(t,v) = (t,e^{s\mathcal{D}^*}v + F_s \Lambda_t \xi), where
\end{equation*}
\begin{equation*}
    F_s = \sum_{j\geq 1} \frac{s^j (\mathcal{D}^*)^{j-1}}{j!}. 
\end{equation*}

\subsubsection{Fixed points of $\mathcal{X}$}

As mentioned in the second section, analyzing observability requires computing the fixed point of the drift.
Considering the matrix of $\mathcal{D}^*$ described by parameters as follows,
\begin{equation*}
    [\mathcal{D}^*] = \begin{bmatrix}
        a & b \\
        c & d
    \end{bmatrix},
\end{equation*}
the the ordinary differential equations system generated by the vector field $\mathcal{X}$ reads as,
\begin{equation*}
    \begin{array}{ccl}
         \dot{t} &=&0  \\
         \dot{x} &=& ax + by + t\xi_1\\
         \dot{y} &=& cx + dy + (e^t-1)\xi_2,
    \end{array}
\end{equation*}
with $\xi = (\xi_1,\xi_2)$. 

To find the fixed points of $\mathcal{X}$, we have to solve the linear system 
\begin{equation}\label{fixpointsXld}
    \begin{array}{rcl}
         ax + by + t\xi_1 &=& 0\\
         cx + dy + (e^t-1)\xi_2 &=& 0.
    \end{array}
\end{equation}

If $\mathcal{D}^*$ is invertible, we get the solution: 
\begin{equation}\label{fixXafim}
    (x_0, y_0) = \left(\frac{-t\xi_1 d + (e^t-1)\xi_2b}{ad-bc}, \frac{-a(e^t-1)\xi_2 + t \xi_1c}{ad-bc}\right). 
\end{equation}

If $\mathcal{D}^*$ is not invertible, $ad=cb$. And, when $a \neq 0$, we obtain:
\begin{equation*}
    x = \frac{-t\xi_1 - by}{a}, 
\end{equation*}
the solution of (\ref{fixpointsXld}), for every $y \in \R$.  

\subsection{Subgroups of $R_2$}
The Lie algebra of dimension 1 is the real vector space $\R$, and the corresponding Lie groups are the simply connected real line $\R$ and the circle $\mathbb{T}$. 

According to \cite[Chapter 7]{onish}, the only real connected Lie groups of dimension 2 are $$\R^2, \R \times \mathbb{T}, \mathbb{T}^2,  \hbox{Aff}(2,\R),$$ where $\hbox{Aff}(2,\R)$ is the affine group, and $\mathbb{T}^n$ is the $n-$torus. 
The simply connected ones are $\R^2$ and $\hbox{Aff}_2(\R)$.  On the other hand, the only two-  dimensional Lie algebras are the Abelian $\R^2$ endowed with the null-bracket and $\mathfrak{r}_2(\R)$, with the Lie bracket given by $$[(t,x),(s,y)] = (0,ty-sx).$$

Let us begin with $R_2$, viewed as the semi-direct product $\R \times_{\rho}\R^2$, following the specific rules:
\begin{equation}\label{prodaffin}
    (t,(x,y)) \cdot (s,(z,w)) = (t + s, (x+z, y + e^t w )). 
\end{equation}

The 2-dimensional subgroups we will consider are the simply connected ones. The structures that are isomorphic to $\R^2$ or $\hbox{Aff}_2(\R)$, can be summarized in the following list:
\begin{eqnarray*}
    G_1 &=&  \R \times_{\rho} (\R \times \{0\}) \simeq \R^2\\
    G_2 &=&  \R \times_{\rho} (\{0\} \times \R) \simeq \hbox{Aff}_2(\R)\\
    G_3 &=&  \{0\} \times_\rho \R^2 \simeq \R^2. 
\end{eqnarray*}

The simply connected 1-dimensional Lie subgroups are described as follows:
\begin{eqnarray*}
    G_4 &=&  \{0\} \times (\{0\} \times \R)\\
    G_5 &=&  \{0\} \times (\R \times \{0\})\\
    G_6 &=&  \R \times (\{0\} \times \{0\}).
\end{eqnarray*}

\subsubsection{The Subgroup $G_1 =  \R \times (\R \times \{0\})$}
It is easy to show that $G_1 = \{ \R \times (\R \times \{0\})\}$ is a subgroup of $\R \times_{\rho}\R^2$. 
Define a homomorphism $h: \R \times_{\rho}\R^2 \longrightarrow G_1$ determined by,
\begin{equation*}
    h(t,(x,y)) = (\alpha_1 t + \alpha_2 x + \alpha_3 y, \beta_1 t + \beta_2 x + \beta_3 y, 0). 
\end{equation*}
Then, $h$ must satisfy:
\begin{equation*}
    h(t,(x,y))h(s,(z,w)) = h(t+s, (x+z, y + e^tw)). 
\end{equation*}

Using the expression provided above and referring to $h$, we get, 
\begin{eqnarray*}
    \alpha_1 (t+s)+\alpha_2 (x+z)+\alpha_3 (y+w) &=& \alpha_1 (t+s)+\alpha_2 (x+z)+\alpha_3 (y+e^t w)\\
    \beta_1 (t+s)+\beta_2 (x+z) + \beta_3 (y + w) &=& \beta_1 (t+s)+\beta_2 (x+z)+\beta_3 (y+e^t w)
\end{eqnarray*}
which implies in $\alpha_3 = \beta_3 = 0$. Therefore, $h$ reads as:
\begin{equation*}
    h(t,(x,y)) = (\alpha_1 t + \alpha_2 x, \beta_1 t + \beta_2 x, 0).
\end{equation*}

The kernel $K$ of the function $h$ is obtained by deriving the following equations, 
\begin{equation}\label{kerafinne}
    \left\{
    \begin{array}{cc}
        \alpha_1 t + \alpha_2 x &= 0\\
        \beta_1 t + \beta_2 x &= 0.
    \end{array}\right.
\end{equation}

Consider the flows: 
\begin{equation}\label{conjugflows}
\begin{array}{ccc}
    \Sigma_1: \varphi_s(t,v) &=& (t,e^{s\mathcal{D}^*}v), \\
    \Sigma_2: \psi_s(t,v) &=& (t,Pe^{s\mathcal{D}^*}P^{-1}v), 
\end{array}
\end{equation}
where $P \in \hbox{GL}_2(\R)$ is a change of basis matrix. We claim that $\Sigma_1$ and $\Sigma_2$ are $\pi-$conjugated in the Lie subgroup $\{0\} \times_{\rho} \R^2$, with $\pi(t,v) = (t,Pv)$. First, we have; 
\begin{equation*}
    \pi(\varphi_s(t,v)) = \pi(t,e^{s\mathcal{D}^*}v) = (t,Pe^{s\mathcal{D}^*}v), 
\end{equation*}
and, 
\begin{equation*}
    \psi_s(\pi(t,v)) = \psi_s(t,Pe^{s\mathcal{D}^*}P^{-1}(Pv)) =(t,Pe^{s\mathcal{D}^*}v).  
\end{equation*}

Moreover, when $t=0$, it turns out that 
\begin{equation*}
    \pi(0,v)\pi(0,w) = (0,Pv+Pw) = (0,P(v+w))=\pi(0,v+w). 
\end{equation*}
Then, $\pi$ is a homomorphism in $\{0\} \times_{\rho} \R^2$. The equality, $$\varphi_t \circ \pi^{-1} = \pi^{-1} \circ \psi_t,$$ is consistent with the same arguments presented earlier, as well as the definition of $$\pi^{-1}(t,v) = (t,P^{-1}v)$$

For an arbitrary matrix $B = \begin{bmatrix}
    \alpha_1 & \alpha_2 \\
    \beta_1 & \beta_2 
\end{bmatrix}$, we can examine the following cases. 

\subsubsection{Case 1: $\det B \neq 0.$} 

The kernel $ h$ denoted by $K$ is given by, 
\begin{equation}\label{kernelDinvert}
    K = \{(t,(x,y)) \in G: t = x = 0\}. 
\end{equation}

Since $t = 0$, it follows that $\Lambda_t = 0.$ Therefore, we get 
\begin{equation*}
    \varphi_s(0,(x,y)) = (0, e^{s\mathcal{D}^*}(x,y)). 
\end{equation*}

According to \cite[Chapter 1]{perko}, the dynamics of $\mathcal{D}^*$ in $\R^2$ are determined by the matrices:
\begin{equation}\label{formsofD}
    \begin{bmatrix}
        \lambda & 0 \\
        0 & \mu
    \end{bmatrix}, \begin{bmatrix}
        \lambda & 1, \\
        0 & \lambda
    \end{bmatrix} \hbox{ and }
    \begin{bmatrix}
        a & -b \\
        b & a
    \end{bmatrix}. 
\end{equation}

Given two matrices $A,B \in \mathfrak{gl}_2(\R)$, let us consider the relationship: 
\begin{equation*}
    A \simeq B \iff PAP^{-1} = B,
\end{equation*}
for some $P \in GL_2(\R)$. 

\begin{proposition}\label{prop1} The linear pair $(\mathcal{X},\pi_K)$ is locally observable if the associated matrix  $\mathcal{D}^*$ determined by the drift, is conjugated to $\begin{bmatrix}
    a & -b \\
    b & a
\end{bmatrix}$ with $b \neq 0$ or $\begin{bmatrix}
    \lambda & 1 \\
    0 & \lambda
\end{bmatrix}$.
\end{proposition}
\begin{proof}

If $[\mathcal{D}^*]  \simeq \hbox{diag}\{\lambda, \mu\}$, we get 
\begin{equation*}
    \varphi_s(0,(x,y)) = (0, e^{s\lambda}x,e^{s\mu}y). 
\end{equation*}

If $e^{s\lambda }x = 0$ for every $s \in \R$ then $x = 0$. Therefore, $I = K$, which shows by the Proposition (\ref{conjug1}), that $(\mathcal{X}, \pi_K)$ is not locally observable. 

If $[\mathcal{D}^*] \simeq \begin{bmatrix}
    \lambda & 1 \\
    0 & \lambda
\end{bmatrix},$
we obtain:
\begin{equation*}
    \varphi_s(0,(x,y)) = (0, e^{s\lambda}(x + sy),e^{s\lambda}y). 
\end{equation*}

If $e^{s\lambda}(x + sy) = 0$ for all $s \in \R$, it follows tha $x = y = 0$. Then $I = \{(0,0,0)\}$, which allows us to conclude,  using the Proposition (\ref{conjug1}), that $(\mathcal{X}, \pi_K)$ is locally observable. 

If $[\mathcal{D}^*] \simeq \begin{bmatrix}
    a & -b \\
    b & a
\end{bmatrix},$ the solution is given by: 
\begin{equation*}
    \varphi_s(0,(x,y)) = (0,e^{as}(\cos{(sb)}x - \sin{(sb)}y),e^{as}(\sin{(sb)}x + \cos{(sb)}y)).
\end{equation*}

Assume $e^{as}(\cos{(sb)}x - \sin{(sb)}y) = 0$ for every $s \in \R$, and $b \neq 0$. By choosing $s = \frac{\pi}{2b}$ and $s = \frac{\pi}{b}$ we obtain $x = 0$ and $y = 0$. Therefore $I = \{(0,0,0)\}$ and by Proposition (\ref{conjug1}), the pair $(\mathcal{X}, \pi_K)$ is locally observable.   Now, if $b = 0$, we are in the diagonal case. Consequently, $K = I$, showing that the system is not locally observable. 
\end{proof}

\subsubsection{Case 2: $\det B = 0.$}

Assume that the first equation in the system  (\ref{kerafinne}) is true. We have several options for the kernel:
\begin{equation*}
    K =  \{(t,(x,y)) \in G: \alpha_1 t = -\alpha_2 x\}.
\end{equation*}
Precisely,
\begin{equation*}
    K  = 
    \left\{
    \begin{array}{ll}
         G&\alpha_1 = \alpha_2 = 0,  \\
         \{(t,(x,y)) \in G: t = 0\}& \alpha_1 \neq 0, \alpha_2 =0, \\
         \{(t,(x,y)) \in G: x = 0\}& \alpha_1 = 0, \alpha_2 \neq 0, \\
         \left\{(t,(x,y)) \in G: t = \frac{-\alpha_2}{\alpha_1} x\right\}& \alpha_1 \neq 0, \alpha_2 \neq 0. 
    \end{array}
    \right.
\end{equation*}

The case $\alpha_1 = \alpha_2 = 0$ implies that $K = G,$ and $(\mathcal{X}, \pi_K)$ is not observable. The case $\alpha_2 \neq 0$ and $\alpha_1 = 0$ will be discussed in the Remark (\ref{example}). For the remaining possibilities, we present the following proposition. 

\begin{proposition}The pair $(\mathcal{X}, \pi_K)$ is not locally observable.     
\end{proposition}

\begin{proof}
First, let us suppose that $\alpha_1 \neq 0$ and $\alpha_2 =0$. The kernel of $h$ is given by
\begin{equation*}
    K = \{(t,(x,y)) \in G: t = 0 \}.
\end{equation*}

So, any point in the form $(0,(x,y))$ belongs to $K$, and the solution $$\varphi_s(0,(x,y)) = (0,e^{s\mathcal{D}^*}(x,y))$$ as well, implying in $I = K$. As $K$ is not discrete, $(\mathcal{X},\pi_K)$ is not locally observable. 

If $\alpha_2 \neq 0$, we obtain
\begin{equation*}
    K =\left\{(t,(x,y)) \in G: t = \frac{-\alpha_2}{\alpha_1}x\right\}. 
\end{equation*}

This means that for every $x \in \R$,  $$\varphi_s(\frac{-\alpha_2}{\alpha_1}x,(x,y))=\left(\frac{-\alpha_2}{\alpha_1}x,e^{s\mathcal{D}^*}(x,y) + F_s \Lambda_{\left(\frac{-\alpha_2}{\alpha_1}x\right)}\xi\right),$$ belongs to $K$. Again, the system in not locally observability.
\end{proof}

\begin{remark}\label{example}
The case $\alpha_1 = 0$ and $\alpha_2 \neq 0$ is unpredictable. Let us discuss some cases. Here,
\begin{equation*}
    K = \{(t,x,y) \in G: x = 0\}. 
\end{equation*}

Consider $\mathcal{D}^*$ in the form
\begin{equation*}
    [\mathcal{D}^*] = \begin{bmatrix}
        0 & 0 \\
        0 & \mu
    \end{bmatrix},
\end{equation*}
for some $\mu \neq 0$. Then 
\begin{equation*}
    e^{s\mathcal{D}^*} = \begin{bmatrix}
        1 & 0 \\
        0 & e^{s\mu}
    \end{bmatrix}. 
\end{equation*}

Hence,
\begin{equation*}
    F_s \Lambda_t = \begin{bmatrix}
        s & 0 \\
        0 & \frac{(e^{\mu s}-1)}{\mu}
    \end{bmatrix}
    \cdot 
    \begin{bmatrix}
        t & 0 \\
        0 & e^t-1
    \end{bmatrix}
    =
    \begin{bmatrix}
        ts & 0 \\
        0 & \frac{(e^t-1)(e^{s\mu}-1)}{\mu}
    \end{bmatrix}. 
\end{equation*}

Thus, 
\begin{equation*}
    \varphi_s(t,0,y) = \left(t,ts\xi_1,e^{s\mu}y + \frac{(e^t-1)(e^{s\mu}-1)}{\mu}\xi_2\right).
\end{equation*}

If $\xi_1 \neq 0$, then $t = 0$, implies that $\varphi_s(0,0,y) = (0,0,e^{s\mu}y)$ still belongs to $K$ and the set $I$ is not discrete. In both cases, $I$ is not discrete. 

If $[\mathcal{D}^*] = \begin{bmatrix}
    0 & 1 \\
    0 & 0 
\end{bmatrix}$, we get 
\begin{equation*}
    e^{s\mathcal{D}^*} = \begin{bmatrix}
        1 & s \\
        0 & 1
    \end{bmatrix}.
\end{equation*}
And,
\begin{equation*}
    F_s \Lambda_t = (sI + \frac{s^2}{2}\mathcal{D}^*)\Lambda_t = \begin{bmatrix}
        st & \frac{s^2}{2}(e^t-1)\\
        0 & s(e^t-1)
    \end{bmatrix}. 
\end{equation*}

Therefore, 
\begin{equation*}
    \varphi_s(t,x,y) = (t,x+sy + st\xi_1 + \frac{s^2}{2}(e^t-1)\xi_2, y + s(e^t-1)\xi_2). 
\end{equation*}
If $(t,x,y) \in K$ such that $\varphi_s(t,x,y) \in K$ for all $s \in \R$, then $x = 0$ and, 
$$sy + st\xi_1 + \frac{s^2}{2}(e^t-1)\xi_2 = 0.$$

Let $\xi \neq 0$. If $t$ and $y$ are non-zero and $s_0,s_1 \in \R\setminus\{0\}$ with $s_1 \neq s_0$, we get:

\begin{equation*}
    y + t\xi_1 + \frac{s_0}{2}(e^t-1)\xi_2 = y + t\xi_1 + \frac{s_1}{2}(e^t-1)\xi_2=0,
\end{equation*}
which must imply $s_0 = s_1$, an absurd. Thus, we conclude: $$y + t\xi_1 + \frac{s}{2}(e^t-1)\xi_2 = 0,$$ if $t = y = 0$, implying $I = \{(0,0,0)\}$. That is, the system is locally observable.

\end{remark}

Considering the previous sections and the expression for the kernel described in (\ref{kernelDinvert}), assume $\mathcal{D}^*$ is invertible. The expression in (\ref{fixXafim}) provides the fixed points of $\varphi$. 

\begin{proposition} Let $\det B \neq 0$ and $\mathcal{D}^*$ is conjugated to $\begin{bmatrix}
    a & -b \\
    b & a
\end{bmatrix}$ with $b \neq 0$ or $\begin{bmatrix}
    \lambda & 1 \\
    0 & \lambda
\end{bmatrix}$. The associated linear pair $(\mathcal{X}, \pi_K)$ is observable.     
\end{proposition}

\begin{proof}The local observability was already proved in the Proposition (\ref{prop1}). By hypothesis $\det B \neq 0$. If $\mathcal{D}^* \simeq\begin{bmatrix}
    \lambda & 1 \\
    0 & \lambda 
\end{bmatrix}$, it follows that,
\begin{equation*}
    Fix (\varphi) \cap K = \{(0,0,0)\},
\end{equation*}
If $t =0$, we get $x_0 = y_0 = 0$, which implies global observability. 

On the other hand, if $\mathcal{D}^* = \begin{bmatrix}
    a & -b \\
    b & a
\end{bmatrix}$ with $b \neq 0$, the intersection between the fixed point of the drift with the kernel of the homomorphism is trivial. Implying by the Proposition (\ref{conjug1}) that the system is observable. 
\end{proof}

\subsubsection{Subgroup $G_2 = \R \times_{\rho} (\{0\} \times \R)$}

Next, we consider the subgroup $G_2 = \R \times_{\rho} \{0\} \times \R$. Let $h: G \longrightarrow G_2$ given by 
\begin{equation*}
    h(t,x,y) = (\alpha_1 t + \alpha_2 x + \alpha_3 y, 0, \beta_1 t + \beta_2 x + \beta_3 y).
\end{equation*}

The homomorphism $h$ must satisfy the following condition:
\begin{equation*}
    h(t,x,y)h(s,z,w) = h(t + s, x+y, y + e^tw). 
\end{equation*}
Therefore, 
\begin{eqnarray*}
    \alpha_1 (t+s)+\alpha_2 (x+z)+\alpha_3 (y+w) &=& \alpha_1 (t+s)+\alpha_2 (x+z)+\alpha_3 (y+e^t w)\\
    \beta_1 t+\beta_2 x + \beta_3 y + e^{\alpha_1 t+\alpha_2 x+\alpha_3 y}( \beta_1 s+\beta_2z + \beta_3 w) &=& \beta_1 (t+s)+\beta_2 (x+z)+\beta_3 (y+e^t w).
\end{eqnarray*}

The first equality gives $\alpha_3 = 0$. If we choose $h$ to be non-zero, the second equality gives $\beta_1 = \beta_2 = 0$ and $\alpha_1 = 1$, $\alpha_2 =0$. Therefore, $h$ has the form: 
\begin{equation*}
    h(t,x,y) = (t,0, \beta_3 y), 
\end{equation*}
with $\beta_3 \neq 0$, for every $(t,x,y), (s,z,w) \in G.$ The case $\beta_3 = 0$ will be discussed in Remark (\ref{beta3=0}). 

The kernel of $h$ reads as,
\begin{equation}
    \ker h = \{(t,x,y) \in G: y = t= 0\}.
\end{equation}

With the previous analysis, we get the following result. 

\begin{proposition}\label{bneq0}Considering the conjugation in (\ref{conjugflows}) and the possible forms of $\mathcal{D}^*$, the only case for $\mathcal{D}^*$ being locally observable is: 
\begin{equation*}
    [\mathcal{D}^*] \simeq\begin{bmatrix}
        a & - b\\
        b & a
    \end{bmatrix}, 
\end{equation*}
with $b \neq 0$. 
\end{proposition}

\begin{proof} Let us suppose at first $\mathcal{D}^*\simeq \diag\{\lambda, \mu\}$. Then 
\begin{equation*}
    \varphi_s(0,x,0) = (0,e^{s\lambda}x,0) \in K, \forall (s,x) \in \R^2,  
\end{equation*}
which shows that $I$ is not discrete. 

If $[\mathcal{D}^*] \simeq 
\begin{bmatrix}
\lambda & 1 \\
0 & \lambda
\end{bmatrix}$, the solution through $(0,x,0)$ is given by 
\begin{equation*}
    \varphi_s(0,x,0) = (0,e^{s\lambda}x,0), 
\end{equation*}
which remains in $K,$ for every pair $(s,x) \in \R^2$. Thus, the system is not locally observable. 

Now, if $[\mathcal{D}^*] \simeq \begin{bmatrix}
    a & -b \\
    b & a
\end{bmatrix}$, the solution in $(0,x,0)$ is given by 
\begin{equation*}
    \varphi_s(0,x,0) = (0,e^{as}\cos{(bs)}x, e^{as}\sin{(bs)}x). 
\end{equation*}

To ensure that $e^{as}\sin{(bs)}x = 0$ for all $s \in \R$, we must consider to possibilities for  $b$. If $b =0$, then we recover the diagonal case. If $b \neq 0$, $\sin{(bs)}x = 0$ for all $s \in \R$ if and only if $x = 0$. Therefore, $I = \{(0,0,0)\}$ if and only if $b \neq 0$, as claimed.  
\end{proof}

Based on the earlier proposition, we can draw the following conclusion.

\begin{proposition}In the same settings of the Proposition (\ref{bneq0}), the pair $(\mathcal{X}, \pi_K)$ is observable. 
\end{proposition}

\begin{proof} Let us consider the expression in (\ref{fixpointsXld}). If $t = 0$, we get $(x_0,y_0) = (0,0)$. We can easily conclude that $Fix(\varphi) \cap K = \{(0,0,0)\}$.    
\end{proof}

\begin{remark}\label{beta3=0}The case when $\beta_3 = 0$ implies that 
\begin{equation*}
    \ker h = \{(t,x,y) \in G: t = 0\}. 
\end{equation*}

The solution is given by, 
\begin{equation*}
    \varphi_s(0,x,y) = (0,e^{s\mathcal{D}^*}(x,y)). 
\end{equation*}
In particular, $I = K$ and consequently, the system can not be locally observable.
\end{remark}

\subsubsection{Subgroup $G_3 = \{0\} \times_{\rho} \R^2$}

Let $h: G \longrightarrow G_3$ be a homomorphism in the form, 
\begin{equation*}
    h(t,x,y) = (0,\alpha_1 t + \alpha_2 x + \alpha_3 y, \beta_1 t + \beta_2 x + \beta_3 y). 
\end{equation*}
Since, $h(t,x,y)h(s,z,w) = h(t+s,x+z,y+e^tw)$, we get: 
\begin{eqnarray*}
    \alpha_1 (t+s)+\alpha_2 (x+z)+\alpha_3 (y+w) &=& \alpha_1 (t+s)+\alpha_2 (x+z)+\alpha_3 (y+e^t w)\\
    \beta_1 (t+s)+\beta_2 (x+z) + \beta_3 (y+w)&=& \beta_1 (t+s)+\beta_2 (x+z)+\beta_3 (y+e^t w).
\end{eqnarray*}

As a consequence, we obtain $\alpha_3 = \beta_3 = 0$. Therefore, 
\begin{equation*}
    h(t,x,y) = (0,\alpha_1 t + \alpha_2 x , \beta_1 t + \beta_2 x). 
\end{equation*}

The kernel of $h$ is given by the solutions of the linear system: 
\begin{equation*}
\left\{
\begin{array}{ccc}
    \alpha_1 t + \alpha_2 x &=& 0\\
    \beta_1 t + \beta_2 x &=& 0.
\end{array}\right.
\end{equation*}

If the matrix $B = \begin{bmatrix}
    \alpha_1 & \alpha_2 \\
    \beta_1 & \beta_2 
\end{bmatrix}$, is invertible, the kernel of $h$ is given by 
\begin{equation*}
    \ker h = \{(t,x,y) \in G: t = x = 0\}. 
\end{equation*}

If $B$ is not invertible, $\alpha_1 \beta_2 = \alpha_2 \beta_1$. Therefore 
\begin{equation*}
    \ker h = \{(t,x,y) \in G: \alpha_1 t = -\alpha_2 x\}. 
\end{equation*}

Next, let us consider the following subcases:
\begin{equation}\label{subcaseskernel}
\ker h= \left\{
    \begin{array}{ll}
         G& \alpha_1 = \alpha_2 = 0.  \\
        \{(t,x,y) \in G: x = 0\}& \alpha_1 = 0, \alpha_2 \neq 0.\\
        \{(t,x,y) \in G: t = 0\} & \alpha_1 \neq 0, \alpha_2 = 0. \\
        \{(t,x,y) \in G: t = \frac{-\alpha_2}{\alpha_1 }x\} & \alpha_1 \neq 0 \neq \alpha_2. 
    \end{array}\right.
\end{equation}

Concerning to local observability, we can derive the following results.
\begin{proposition}If $B$ is invertible, $(\mathcal{X}, \pi_K)$ is locally observable when $[\mathcal{D}^*]$ is conjugated to 
\begin{equation*}
    \begin{bmatrix}
        \lambda & 1 \\
        0 & \lambda
    \end{bmatrix} \hbox{ or } \begin{bmatrix}
        a & - b\\
        b & a
    \end{bmatrix}, b \neq 0. 
\end{equation*}    
\end{proposition}

\begin{proof}If $\mathcal{D}^* \simeq \begin{bmatrix}
    \lambda & 0 \\
    0  & \mu
\end{bmatrix}$, we have 
\begin{equation*}
    \varphi_s(0,0,y) = (0,0,e^{s\mu}y) \in K, \forall s \in \R, 
\end{equation*}
which implies that $I$ is not discrete. 

If $[\mathcal{D}^*] \simeq \begin{bmatrix}
    \lambda & 1 \\
    0 & \lambda
\end{bmatrix}$, we obtain  
\begin{equation*}
    \varphi_s(0,0,y) = (0,e^{\lambda s}sy, y). 
\end{equation*}
If $e^{\lambda s}sy = 0$ for all $s \in \R$, then $y = 0$. Therefore, $I = \{(0,0,0)\}$
proving local observability. 

If $[\mathcal{D}^*] \simeq \begin{bmatrix}
    a & -b \\
    b & a
\end{bmatrix}$, with $b \neq 0$,  then
\begin{equation*}
    \varphi_s(0,0,y) = (0,-e^{as}\sin{(sb)}y, e^{as}\cos{(bs)}y).
\end{equation*}

If $\varphi_s(0,0,y) \in K$ for all $s \in \R$, we  get $e^{as}\sin{(sb)}y = 0$, for all $s \in \R.$ That is, $\sin{(sb)}y = 0,$ for all $s \in \R.$ Choosing $s = \frac{\pi}{2b}$, we get $y = 0$. Therefore, $I = \{(0,0,0)\}$, and the system is locally observable. If $b = 0$, we recover the diagonal case, which indicates that the system is not locally observable.
\end{proof}

Due to matrix B not being invertible, we can draw the following conclusion.
\begin{proposition}\label{t=0}Assume $B$ is not invertible. In cases $1$, $2$ and $4$ for $\ker h$ in (\ref{subcaseskernel}), the pair $(\mathcal{X}, \pi_K)$ is not locally observable.     
\end{proposition}

\begin{proof}The case $\alpha_1 = \alpha_2 = 0$ is trivial. Considering $\alpha_1 \neq 0$ and $\alpha_2 = 0$, we have,
\begin{equation*}
    \varphi_s(0,x,y)  = (0,e^{s\mathcal{D}^*}(x,y)) \in K, \forall s \in \R. 
\end{equation*}

Therefore, $I = \{0\} \times \R^2$, which is not discrete. Finally, if $\alpha_1 \neq 0$ and $\alpha_2 \neq 0$, considering the point $\left(\frac{-\alpha_2}{\alpha_1}x,x,y\right) \in K$. It turns out,
\begin{equation*}
    \varphi_s\left(\frac{-\alpha_2}{\alpha_1}x,x,y\right) =\left(\frac{-\alpha_2}{\alpha_1}x,e^{s\mathcal{D}^*}(x,y) + F_s \Lambda_{\left(\frac{-\alpha_2}{\alpha_1}x\right)}\xi\right) \in K, \forall s \in \R,  
\end{equation*}
which also implies $I = K$, ending the proof. 
\end{proof}

The case $\ker h = \{(t,(x,y)) \in G: x = 0\}$ is discussed in the Remark (\ref{example}). \\
When $B$ is invertible, we establish the following proposition. 

\begin{proposition}The pair $(\mathcal{X},\pi_K)$ is observable if $\mathcal{D}^*$ is conjugated to 
\begin{equation}\label{lastprop}
    \begin{bmatrix}
        \lambda & 1 \\
        0 & \lambda 
    \end{bmatrix} \hbox{ or }\begin{bmatrix}
        a & -b \\
        b & a
    \end{bmatrix}, b \neq 0. 
\end{equation}
\end{proposition}

\begin{proof} As a matter of fact, $$\ker h = \{(t,x,y) \in G: t=x=0\}.$$ So, $Fix(\varphi) = \{(0,0,0)\}$ if $t = 0$. And, $K \cap Fix(\varphi) = \{(0,0,0)\}$ for both matrices in (\ref{lastprop}).     
\end{proof}

\subsubsection{Other subgroups}

Consider the group $G_4 = \{0\} \times \{0\} \times \R$. For the homomorphism $h: G \longrightarrow G_4$, we get 
\begin{equation*}
    h(t,x,y) = (0,0,\alpha t + \beta x). 
\end{equation*}
Therefore,
\begin{equation}
\ker h = \left\{
    \begin{array}{ll}
         G& \alpha = \beta = 0.  \\
        \{(t,x,y) \in G: x = 0\}& \alpha = 0, \beta \neq 0.\\
        \{(t,x,y) \in G: t = 0\} & \alpha \neq 0, \beta = 0. \\
        \{(t,x,y) \in G: t = \frac{-\beta}{\alpha}x\} & \alpha \neq 0 \neq \beta.
    \end{array}\right.
\end{equation}
The cases mentioned are the same as those in (\ref{subcaseskernel}). In particular, for the cases $1$, $2$, and $4$ the corresponding pairs are not locally observable. The case $3$ depends on the matrix $\mathcal{D}^*$ (see Remark (\ref{example})). It is not hard to conclude the same for the subgroups $G_5 = \{0\} \times \R \times \{0\}$ and $G_6 = \R \times \{0\}\times\{0\}$.

\subsection{Subgroups of $R_3$}

Next we deal with $R_3$, which is the semi-direct product $\R\rtimes_{\rho} \R^2$, endowed with the product: 
\begin{equation*}
    (t,x,y) \cdot (s,z,w) = (t +s, x + e^t(z +tw), y +e^tw). 
\end{equation*}

Using the same reasoning applied in $R_2$, the subgroups of $R_3$ are as follows:
\begin{eqnarray*}
    G_1 &=& \R \times_{\rho} \R\times \{0\}\\
    G_2 &=& \{0\} \times_{\rho} \R^2 \\
    G_3 &=& \{0\} \times_{\rho} \{0\}\times \R\\
    G_4 &=& \{0\} \times_{\rho} \R \times \{0\}\\
    G_5 &=& \R \times_{\rho} \{0\}^2.
\end{eqnarray*}

We notice that $G_1$ is isomorphic to $\hbox{Aff}(2,\R)$, $G_2$ to $\R^2$, $G_3, G_4$ and $G_5$ are isomorphic to $\R$. 

\subsubsection{The Subgroup $G_1$}

Considering $h: R_3 \longrightarrow G_1$ a homomorphism in the form 
\begin{equation*}
    h(t,x,y) = (\alpha_1 t +\alpha_2 x +\alpha_3 y, \beta_1 t +\beta_2 x +\beta_3 y, 0), 
\end{equation*}
by the expression 
\begin{equation*}
    h(t,x,y)h(s,z,w) = h(t+s,x+e^{t}(z + tw), y +e^t w), 
\end{equation*}
we obtain $\alpha_1 = 1, \alpha_2 = \alpha_3 = \beta_1 = \beta_2 = 0$. Therefore, 
\begin{equation*}
    h(t,x,y) = (t, \beta_3 y,0),
\end{equation*}
whose kernel is given by 
\begin{equation*}
    \ker h= \left\{\begin{array}{lc}
         \{(t,x,y) \in R_3:t = y = 0\},& \beta_3 \neq 0  \\
         \{(t,x,y) \in R_3: t = 0\},& \beta_3 = 0. 
    \end{array}\right.
\end{equation*}

The case $\beta_3 \neq 0$ is explained in the Proposition (\ref{bneq0}). The case $\beta_3 = 0$ is analyzed in Proposition (\ref{t=0}). 

\subsubsection{The Subgroup $G_2$}

Consider $h:R_3 \longrightarrow G_2$, in the form
\begin{equation*}
    h(t,x,y) = (0,\alpha_1 t +\alpha_2 x +\alpha_3 y, \beta_1 t +\beta_2 x +\beta_3 y).
\end{equation*}
Hence, $\alpha_2 = \alpha_3 = \beta_2 = \beta_3 = 0$. Thus, 
\begin{equation}\label{G_2R_2}
    h(t,x,y) = (0,\alpha_1 t, \beta_1 t). 
\end{equation}

 The kernel of $h$ reads as 
\begin{equation}\label{t=000}
    \ker h = \left\{\begin{array}{cc}
         R_3,& \alpha_1 = \beta_1 = 0.  \\
         \{(t,x,y) \in R_3: t=0\},& \hbox{ otherwise.} 
    \end{array}\right.
\end{equation}
which also leads to the Proposition (\ref{t=0}). 

\subsubsection{The Subgroup $G_3$}

The homomorphisms between $R_3$ and $G_3$ are given by 
\begin{equation*}
    h(t,x,y) = (0,0,\alpha t). 
\end{equation*}
Here, the kernel is as in (\ref{t=000}). The same behavior observed with homomorphisms also applies to the subgroups $G_4,G_5$ and $G_6$. 

\subsection{The Subgroups of $R_{3,\lambda}$}

The Lie group $R_3$ is the set $\R \times_{\rho}\R^2$ endowed with the semi-direct product 
\begin{equation*}
    (t,x,y) \cdot (s,z,w) = (t + s, x +e^tz, y +e^{\lambda t}w). 
\end{equation*}

Let us consider the subgroups: 
\begin{eqnarray*}
    G_1 &=& \{0\} \times \R^2\\
    G_2 &=& \R \times \R \times \{0\}\\
    G_3 &=& \R\times \{0\} \times \R\\
    G_4 &=& \R\times \{0\}^2\\
    G_5 &=& \{0\} \times \R \times \{0\}\\
    G_6 &=& \{0\} \times \{0\} \times \R.
\end{eqnarray*}
Obviously, $G_4, G_5$ and $G_6$ are isomorphic to $\R$. The subgroup $G_2$ is diffeomorphic to the set $\hbox{Aff}(2,\R)$ and $G_1$ is diffeomorphic to $\R^2$. 

Finally, the subgroup $G_3$ can also be considered as the plane $\R^2$ endowed with the product, $$(t,x)\cdot(s,y) = (t +s, x +e^{\lambda t}y).$$ If $h: \hbox{Aff}(2,\R) \longrightarrow G_3$, is defined by: $$h(t,x) = (\lambda^{-1} t, x),$$ it is not hard to prove that $h$ is an isomorphism. Therefore, $G_3$ is also isomorphic to $\hbox{Aff}(2,\R)$. 

We conclude this section by examining the observability properties through the previously listed subgroups.

\subsubsection{The Subgroup $G_1$}

Considering the same pattern for the homomorphisms between the group and its subgroups, the homomorphism $h:R_{3,\lambda} \longrightarrow G_1$ is given by, 
\begin{equation*}
    h(t,x,y) = (0,\alpha_1 t, \beta_1 t), 
\end{equation*}
which is the same case for the homomorphism in (\ref{G_2R_2}). 

\subsubsection{Subgroup $G_2$}

The homomorphisms between $R_{3,\lambda}$ and $G_2$ are given by functions $h: R_{3,\lambda} \longrightarrow G_2$ defined by 
\begin{equation*}
    h(t,x,y) = (t,\beta_2 x,0). 
\end{equation*}

Therefore, 
\begin{equation*}
    \ker h = \left\{\begin{array}{cc}
         \{(t,x,y) \in R_{3,\lambda}:t = x = 0\},& \beta_2 \neq 0  \\
         \{(t,x,y) \in R_{3,\lambda}:t=0\}& 
    \end{array}\right.
\end{equation*}

The case $\beta_2 \neq 0$ was explored in the expression (\ref{kernelDinvert}) and the Proposition (\ref{prop1}). 

\subsubsection{The Subgroup $G_3$}

The subgroup $G_3$ is isomorphic to $G_2$. Hence, taking $$h(t,x,y) = (\lambda^{-1}t,\beta_2x,0)$$ the cases for the kernel of $h$ are the same as in the $G_2$ case. 

\subsubsection{The Subgroup $G_4$}

The homomorphisms are given by the functions $h: R_{3,\lambda} \longrightarrow G_4$ defined by, 
\begin{equation*}
    h(t,x,y) = (\alpha_1 t,0,0). 
\end{equation*}
It turns out that,
\begin{equation*}
    \ker h = \left\{\begin{array}{lc}
         R_{3,\lambda},& \alpha_1 = 0.  \\
         \{(t,x,y) \in R_{3,\lambda}: t=0\},& \alpha_1 \neq 0. 
    \end{array}
    \right.
\end{equation*}
This led to the cases we have previously discussed. It is easy to conclude that the subgroups $G_5$ and $G_6$ share the same results.

\subsection{The Subgroups of $R'_{3,\lambda}$ and $E$}

Here, we will deal only with $R'_{3,\lambda}$, since $E$ has the same structure with $\lambda = 0$. 

It is easy to confirm that the only existing subgroups are as follows:
\begin{eqnarray*}
    G_1 &=& \{0\} \times \R^2, \\
    G_2 &=& \R \times \{0\}^2. 
\end{eqnarray*}

Just observe that $\{0\} \times \R \times \{0\}$ and $\{0\} \times \{0\} \times \R$, are the same as $G_2$.

\subsubsection{The Subgroup $G_1$}

The homomorphisms $h: R_{3, \lambda}' \longrightarrow G_1$ $$h(t,x,y) = (0,\alpha_1 t + \alpha_2 x + \alpha_3 y, \beta_1 t + \beta_2 x + \beta_3 y),$$ reads as:
\begin{equation*}
    h(t,x,y) = (0,\alpha_1 t,\beta_2 t), 
\end{equation*}
which follows the form outlined in equation (\ref{G_2R_2}) and has already been studied. 

\subsubsection{The Subgroup $G_2$}

For this case, the standard form is given by, $$h(t,x,y) = (\alpha_1 t + \alpha_2 x + \alpha_3 y, 0,0).$$ By applying the product rule, we find $\alpha_2 = \alpha_3 = 0$, resulting in a scenario similar to  (\ref{t=000}). 

\section{Acknowledgements}
The author Alexandre J. Santana is partially supported by CNPq grant n. 309409/2023-3.

\section{Conclusion}

This paper examines the conditions for observability in non-nilpotent solvable three-dimensional Lie groups. It focuses on the relationships between the group structure, the kernels of homomorphisms involving simply connected subgroups, and the dynamics of linear vector fields. By combining the analysis of Lie group homomorphisms with the properties of the drift, we identified criteria that depend not only on the algebraic structure of the Lie group but also on the specific characteristics of the linear vector field being considered. 

We also explored the conditions under which observability is unattainable, providing specific examples related to the corresponding derivation of the drift. This study proposes new avenues for future research, such as investigating observability in higher-dimensional or more complex Lie groups and developing computational methods to apply these criteria in practical scenarios.

\end{document}